\documentclass{amsart}
\usepackage{amssymb,amscd,color}
\usepackage{va}
\usepackage{tikz}
\usepackage{graphicx}
\usepackage{enumitem}
\usepackage[normalem]{ulem}
\usepackage{xcolor}
\usepackage{bm}
\usepackage{tikz, tikz-cd}

\usepackage[
urlcolor=blue, linktocpage, hyperindex=true, colorlinks=true,
linkcolor=blue,
citecolor=blue,
]{hyperref}
\newcommand\Diff{\operatorname{Diff}}

\newcommand\Acc{\operatorname{Acc}}
\newcommand\dedge{_{\rm edge}}

\newcommand\dsing{_{\rm sing}}

\newcommand{\PP}{\mathbb{P}}
\newcommand{\Vis}{\operatorname{Vis}}
\newcommand{\ZZ}{\mathbb{Z}}

\begin{document}
\title[Surfaces from four lines]{Log surfaces of Picard rank one\\ from four lines in the plane}
\author{Valery Alexeev}
\address{Valery Alexeev\\Department of Mathematics\\ University of Georgia\\ Athens\\GA 30605\\ USA}
\email{valery@math.uga.edu}
\author{Wenfei Liu}
\address{Wenfei Liu \\School of Mathematical 
Sciences\\ Xiamen University\\Siming South Road 422\\Xiamen\\ Fujian 361005\\ P.~R.~China}
\email{wliu@xmu.edu.cn}
\subjclass[2010]{Primary 14J29; Secondary 14J26, 14R05}
\date{January 25, 2018}
\keywords{log canonical surfaces, volume}
\thanks{The first author was supported by NSF grant DMS-1603604. The second author was partially supported by the NSFC (No.~11501012 and No.~11771294)}

\begin{abstract}
  We derive simple formulas for the basic numerical invariants of a 
  singular surface with Picard number one obtained by blowups and
  contractions of the four-line configuration in the plane. As an
  application, we establish the smallest positive volume and the
  smallest accumulation point of volumes of log canonical surfaces
  obtained in this way.
\end{abstract}
\maketitle
\tableofcontents

\section{Introduction}

Let $X$ be a projective normal surface and $B=\sum b_iB_i$ be an $\bR$-divisor
with coefficients in a DCC set, i.e.\ one satisfying the
descending chain condition.  Assume that the pair $(X,B)$ has log
canonical singularities and that the log canonical divisor $K_X+B$ is
ample.  It is known from \cite{alexeev1994boundedness-and-ksp-2} that
the set of volumes $(K_X+B)^2$ is also a DCC set and thus attains the
absolute minimum, a positive real number. The paper
\cite{alexeev2004bounding-singular} gives an effective lower bound for it
which however is unrealistically small.

In \cite{alexeev2016open-surfaces} we found surfaces with the smallest
known volumes for the sets $\cS_0=\{0\}$ and $\cS_1=\{0, 1\}$.  In
\cite{alexeev2018accumulation-points} we proved reasonable lower
bounds for the accumulation points of the sets of these volumes.  All
of the best known examples (including for other common DCC sets $\cS$)
are based on the following construction which despite its simplicity
is expected to be optimal.

\begin{construction}\label{con:main}
Let $L_1,L_2,L_3,L_4$ be four lines in general position in the
projective plane. Consider a diagram

\begin{center}
\begin{tikzcd}
& Y \arrow[dl, swap, "f"] \arrow[dr, "g"] \\
\bP^2 && X
\end{tikzcd}  
\end{center}
in which $f\colon Y\to\bP^2$ is a sequence of blowups at the points of
intersection between the curves which are either strict preimages
of $L_k$ or are exceptional divisors $E_j$. We will call these
curves \emph{the visible curves}. The morphism $g\colon Y\to X$ is
a contraction of some of the visible curves to a normal surface
$X$. The images $B_i$ of the non-contracted visible curves are called
\emph{survivors}. We will consider pairs $(X,B)$, where
$B=\sum b_iB_i$ is a linear combination of survivors.
\end{construction}

In this paper we tackle the case when the Picard number $\rho(X)$ is 1,
i.e.\ when there are exactly 4 survivors. In
this case the record surface in \cite{alexeev2016open-surfaces} for
the sets $\cS_0$ and $\cS_1$ has volume $\frac1{6351}$. 
Here we prove that
this bound is optimal for the sets $\cS_0,\cS_1$ and Picard number~1,
for the surfaces in Construction~\ref{con:main}.
We also prove that the minimum of the limit points of these volumes is
$\frac1{78}$. 
(Note however that the absolute champions in
\cite{alexeev2016open-surfaces} have $\rho=2$.)

The main contribution of this paper is a simple explicit formula for
$(K_X+B)^2$ which we then apply. This formula works without assuming
that $K_X+B$ is ample or that $(X,B)$ has log canonical singularities,
and may be used in other situations, for example for log del Pezzo
and log Calabi-Yau (or Enriques) surfaces.

Dongseon Hwang has informed us that he ran some computer experiments
for surfaces of Picard number 1 that did not yield a volume better
than $\frac1{6351}$ found in \cite{alexeev2016open-surfaces}. Actually
proving this bound, for example in the way we did it in
\cite[Thm. 8.2]{alexeev2016open-surfaces} for a special case, was the
main motivation for this paper.

\section{Surface singularities and their determinants}
\label{sec:singularities}

Let $X$ be a normal surface and $B=\sum b_iB_i$ be a $\bQ$- or
$\bR$-divisor with coefficients $0\le b_i\le 1$. Let $f\colon Y\to X$
be a resolution of singularities with a normal crossing divisor
$f_*\inv B \cup \Exc(\pi)$. Consider the natural formula
\begin{displaymath}
  K_Y  = f^*(K_X+B) + \sum a_j E_j.
\end{displaymath}
Here, the divisors $E_j$ are both the $f$-exceptional divisors and the
strict preimages of the divisors $B_i$; for the latter 
 one has
$a_i = -b_i$.  The numbers $a_j$ are called discrepancies,
$c_j= 1+a_j$ are log discrepancies, and $b_j=-a_j=1-c_j$ are
codiscrepancies.  The pair $(X,B)$ is called log canonical or lc
(resp. Kawamata log terminal or klt) if all $a_j\ge-1$,
i.e.\ $c_j\ge0$ (resp. $a_j>-1$, $c_j>0$, $b_j<1$).
One says that $(X, B)$ is canonical at a point $p\in X$ if the
discrepancy for any exceptional divisor over $p$ is nonnegative. 

Log canonical singularities of surfaces in any characteristic are
classified by their dual graphs, cf.
\cite{alexeev1992log-canonical-surface}. When $B=0$ the answer is as
follows.

\begin{definition}
  Let $f\colon\wX\to (X,0)$ be the minimal resolution and $E_i$ be the
  $f$-exceptional divisors. The \emph{dual graph} has a vertex for
  each curve $E_i$, marked by a positive integer $n_i=-E_i^2$. Two
  vertices are connected by 
  $E_i.E_j$
  edges.  Vice versa, each marked
  multigraph gives a quadratic form $(-E_i.E_j)$. For simplicity we
  always work with the negative of the intersection matrix
  since it is positive definite.  The diagonal entries of such a
  matrix are $>0$ and the off-diagonal entries are $\le0$.
\end{definition}

Then, first of all, singularities corresponding to arbitrary chains
$[n_1,\dotsc, n_k]$ with $n_i\ge2$ are klt. These are in a bijection
with rational numbers $0<\frac{p}{q}<1$ via the Hirzebruch-Jung (HJ)
continued fractions
\begin{displaymath}
  \frac{q}{p} = n_1 - \frac1{ n_2 - \frac1{ n_3 - \dots} }.
\end{displaymath}
Here, $q=\det (-E_i.E_j)$ is the determinant of the matrix with the
diagonal entries $n_i$ and with $-1$ on the diagonals adjacent to it.
Our notation for this determinant is $|n_1, n_2, \dotsc, n_k|$.

\begin{remark}\label{rem:generalized-HJ}
  We will need a slight generalization of this constructions, as
  follows. Let $\frac{p}{q}$ be a fraction \emph{larger than or equal to 1}, so
  that $\frac{q}{p} \le 1$. Then by the same continued fraction
  expansion it corresponds to the chain $[1, n_2, \dotsc, n_k]$. In
  this way, we get a bijection between the positive rational numbers
  $\frac{p}{q}$ and the chains $[n_1, n_2, \dotsc, n_k]$ in which the
  starting number is $n_1\ge1$ and all others are $n_i\ge2$.
\end{remark}

In addition to the chains, graphs with a positive definite quadratic
form and a single fork from which three chains with determinants
$q_1,q_2,q_3$ emanate are klt iff
$\frac1{q_1}+\frac1{q_2} + \frac1{q_3} > 1$, and they are log
canonical iff $\frac1{q_1}+\frac1{q_2} + \frac1{q_3} \ge 1$. The
possibilities for $(q_1,q_2,q_3)$ are $(2,2,n)$, $(2,3,3)$, $(2,3,4)$,
$(2,3,5)$, $(3,3,3)$, $(2,4,4)$, $(2,3,6)$ and correspond to the Lie
types $D_{n+2}$, $E_6$, $E_7$, $E_8$, $\wE_6$, $\wE_7$, $\wE_8$.
There is also a graph of type $\wD_n$ with two forks and four legs
with determinants $(2,2,2,2)$. Finally, there is graph of type $\wA_n$
which is a cycle.
For the $\wA_1$ and $\wA_2$ graphs, the curves $E_i$ should intersect
at distinct point: 
tacnode and triple points are not allowed.
We denote the determinant of the cycle with marks $n_i$ by
$|n_1, n_2, \dotsc, n_k, \circlearrowleft|$.

\begin{remark}
  When all the marks are $n_i=2$, the above are the dual graphs of Du
  Val singularities and of Kodaira's degenerations of elliptic curves.
  But here any marks $n_i$ are allowed, as long as the form
  $(-E_i.E_j)$ is positive definite.
\end{remark}

For the set $\cS_1=\{0, 1\}$ and $B\ne0$, i.e.\ when $B$ is nonempty and
reduced, one can have $B$ to be attached to one or both ends of a
chain, and to a leg of a $D_n$ graph if the other legs have
$(q_1,q_2)=(2,2)$. A degenerate case of this is attaching $B$ to the
middle of a chain $[2,n,2]$, this is also allowed.
This completes the list.

\begin{lemma}\label{lem:det-expansion}
  Divide the set of vertices 
  $V(\Gamma)$ into two disjoint subsets $V_1\sqcup V_2$, and let
  $\Gamma_i$ will be the induced subgraphs on these vertex
  sets. Assume that there are no cycles in $\Gamma$ involving both
  $\Gamma_1$ and $\Gamma_2$, in other words $h_1(\Gamma) =
  h_1(\Gamma_1) + h_1(\Gamma_2)$, where $h_1$ denotes the rank of the
  first homology group of a connected graph. 
  Then
  \begin{displaymath}
    \det \Gamma = \det\Gamma_1 \cdot \det\Gamma_2 +
    (-1)^k
    \hskip -22pt
    \sum_{(u_1,v_1),\dotsc,(u_k,v_k)}
    \hskip -22pt
    \det(\Gamma_1 - u_1 \dotsc -u_k) \cdot
    \det(\Gamma_2 -v_1 \dotsc - v_k), 
  \end{displaymath}
  in which the sum goes over collections of edges $(u_i,v_i)$ in $\Gamma$ with
  $u_i\in\Gamma_1$, $v_i\in\Gamma_2$. 
\end{lemma}

\begin{proof}
  In the expansion of $\det\Gamma$ into the sum of $n!$ products, the
  terms which are not listed in the above formula correspond to paths that
  enter from $\Gamma_1$ into $\Gamma_2$ and then eventually die, as
  there are no cycles coming back.
\end{proof}

\begin{corollary} One has
  $|n_1, n_2, \dotsc, n_k| = n_1 \cdot |n_2, \dotsc, n_k| - |n_3 \dotsc,
  n_k|$ and \newline
  $|n_1,\dotsc, n_i,n_{i+1}, \dotsc  n_k| =
  |n_1,\dotsc, n_i| \cdot |n_{i+1}, \dotsc  n_k| -
  |n_1,\dotsc, n_{i-1}| \cdot |n_{i+2}, \dotsc  n_k|.$
\end{corollary}

Here, by convention, the determinant of a chain of length $k=0$ is
1. There is a generalization of \eqref{lem:det-expansion} when there
are cycles between $\Gamma_1,\Gamma_2$. We will not need it since the
only non-tree log canonical graph is a cycle. For it, it is easy to
prove 
\begin{displaymath}
  |n_1, n_2, \dotsc, n_k, \circlearrowleft| =
  n_1 \cdot |n_2, \dotsc, n_k| -
  |n_3, \dotsc, n_k| - |n_2, \dotsc, n_{k-1}| - 2,
\end{displaymath}
with 2 accounting for the two directed cycles in $\Gamma$, clockwise
and counter clockwise.

\begin{corollary}\label{cor:attaching-2s}
  The determinant of a graph $[(m-1)\times 2] \cup \Gamma$ obtained by
  attaching the chain $[(m-1)\times 2] = [2,\dotsc, 2]$ of $(m-1)$ 2's to a
  vertex $v$ with mark $n$ is
  \begin{displaymath}
    \det\, [(m-1)\times 2] \cup [n] \cup \Gamma'  =
    - \det\,  [-m, n-1] \cup \Gamma' .
  \end{displaymath}
  where 
  $\Gamma'= \Gamma - v$ and $[-m, n-1] \cup \Gamma' $
  is the graph obtained by replacing the mark $n$ of
  $v\in\Gamma$
    by $n-1$ and attaching a single vertex marked $-m$.
\end{corollary}
\begin{proof}
  Applying~\ref{lem:det-expansion} twice gives
  \begin{eqnarray*}
     \det\, [(m-1)\times 2] \cup [n] \cup \Gamma'  &=&
      m  \det\, [n] \cup \Gamma \ -\  (m-1)\det\,\Gamma' \quad\text{and}\\
     \det\,  [-m, n-1] \cup \Gamma'   &=&
      (-m) \det\,  [n-1] \cup \Gamma' \ -\  \det\, \Gamma' \\
    &=& (-m) \det\,  [n] \cup \Gamma'  + m \det\, \Gamma' - \det
        \Gamma'.
  \end{eqnarray*}
\end{proof}

Below, we will need to deal with the following situation. Let
$\bm{\Gamma}(n_i)$ be a  ``core graph'' with the vertices marked
$n_i$. On top of each vertex $v_i$ we ``graft'' several chains corresponding
to HJ fractions
$\frac{q_{ij}}{p_{ij}}$ as in Remark~\ref{rem:generalized-HJ}. We
emphasize that we \emph{do not attach} a chain. Instead,
\emph{grafting} means that we put an end of the chain for the fraction
$\frac{q_{ij}}{p_{ij}}$ \emph{on top} of the vertex $v_i$. Thus, if
\begin{displaymath}
  \frac{p_{ij}}{q_{ij}} = n_{ij} - \frac{r_{ij}}{q_{ij}},
  \qquad p_{ij} = n_{ij} q_{ij} - r_{ij}
\end{displaymath}
then the mark of the vertex $v_i$ in the new graph is $n_i + \sum_j n_{ij}$ and
the legs have determinants~$q_{ij}$. The legs ``attached'' to the core
correspond to the fractions $\frac{r_{ij}}{q_{ij}}$.

We will call thus obtained graph $\wh\Gamma\big( n_i;
\frac{q_{ij}}{p_{ij}} \big)$ a ``hairy graph'', with 
hairs being the chains coming out of the vertices of the core graph.

\begin{theorem}\label{thm:hairy-graph}
  The determinant of a hairy graph can be computed by the formula
  \begin{displaymath}
    \det\wh\Gamma \big( n_i; \frac{q_{ij}}{p_{ij}} \big) =
    \det\bm{\Gamma} \big(n_i + \sum_j \frac{p_{ij}}{q_{ij}} \big)
    \cdot \prod_{i,j} q_{ij} 
  \end{displaymath}
  where the core graph $\bm{\Gamma(n_i)}$ has new marks
  $\bm{n_i} = n_i + \sum_j \frac{p_{ij}}{q_{ij}}$. Alternatively,
  \begin{displaymath}
    \det\wh\Gamma \big( n_i; \frac{q_{ij}}{p_{ij}} \big) =
    \det\bm{\widetilde\Gamma} \big(n_i; -\frac{q_{ij}}{p_{ij}} \big)
    \cdot \prod_{i,j} (-p_{ij})
  \end{displaymath}
  where the graph
  $\bm{\widetilde\Gamma}$ is obtained from the graph $\Gamma(n_i)$ by
  adding a single vertex of weight $-\frac{q_{ij}}{p_{ij}}$ for each hair.
\end{theorem}
\begin{proof}
  Follows by repeatedly applying Lemma~\ref{lem:det-expansion}. 
\end{proof} 

\begin{example}\label{ex:attaching-2s-again}
  Attaching the chain $[(m-1)\times 2]$ to a vertex with mark $n$ is
  the same as grafting the chain $[1, (m-1)\times 2]$ onto a vertex with
  mark $(n-1)$.
  The chain $[1, (m-1)\times 2]$ corresponds to the
  HJ fraction $\frac{q}{p} = \frac{m}{1}$.
  % because $\frac1{m} = 1 - \frac{m-1}{m}$, and the HJ fraction
  % $\frac{m-1}{m}$ corresponds to $[(m-1)\times 2]$.
  The second formula of \eqref{thm:hairy-graph} is now precisely 
  \eqref{cor:attaching-2s}. The first formula of
  \eqref{thm:hairy-graph} gives an alternative expression for this
  determinant as 
  $m \det\bm{\Gamma}$, where $\bm{\Gamma}$ is the core graph with the
  weight ${\bm{n}} = n-1+\frac1{m}$ at the vertex $v$.

\end{example}

\section{Weight vectors of visible curves and weight matrices}
\label{sec:weight-matrices}
 
We follow the notations of Construction~\ref{con:main}. Here, we encode
each visible curve uniquely by a \emph{weight vector} in $\bZ^4$, and
the entire surface $X$ by a \emph{weight matrix}.

\begin{definition}
  The weight vector of a line $L_i$ is the vector $e_i$ in the
  standard Euclidean basis of $\bZ^4$. For an exceptional curve $E$ of
  $f$ its weight vector is $(w_1,w_2,w_3,w_4)$, where $w_i$ is the
  coefficient of $E$ in the full pullback $f^*(L_i)$.
\end{definition}

In our situation, every visible curve other than $f_*\inv L_i$ lies
over the intersection of exactly two lines, say $L_i\cap L_j$. For it,
$w_i>0$, $w_j>0$ and $w_k=0$ for $k\ne i,j$. 
We can identify a weight vector with $w_iw_j\neq 0$ to an element in
$\ZZ_i\times \ZZ_j\cong \ZZ^2$ where $\ZZ_i$ and $\ZZ_j$ are the
$i$-th and $j$-th factors of $\ZZ^4$ respectively.

\begin{definition}\label{def:weight-matrix}
  The \emph{weight matrix $W$} of a surface $X$ is an $N \times 4$
  matrix whose rows are the weight vectors of the survivors
  $E_1, \dotsc, E_N$, where $N=\rho(X)+3$.  Thus, the four columns of
  $W$ are the pullbacks $f^*(L_i)$ 
  for the four lines, written as
  linear combinations of the visible curves, with all but coefficients
  in $E_i$ ignored.

  The \emph{reduced weight matrix $\oW$} is an $N\times 3$ matrix
  with columns $f^*(L_i-L_4)$.
\end{definition}

\begin{definition}\label{def:extended-weight-matrix}
  Given a pair $(X,B = \sum b_iB_i)$ as in
  Construction~\ref{con:main}, the \emph{extended weight matrix} $\hW$
  is an $(N+1) \times 5$ matrix whose entries in the last column are
  the log discrepancies $c_i=1-b_i$, and with the row $(1,\dotsc, 1)$
  added at the bottom.
\end{definition}

Note that $W$ and $\hW$ are square matrices iff $\rho(X)=1$,
of sizes $4\times 4$ and $5\times 5$ respectively.
We now establish a description of the dual graph of the visible curves
on $Y$ in terms of the weight vectors. We begin with the following
situation.

\begin{construction}\label{con:single-curve}
  Let $L_i,L_j$ be two smooth curves on a smooth surface $S$,
  intersecting normally at a single point~$P$. These are not
  necessarily lines in $\bP^2$; we will later apply this to the
  lines. Since there is only one point, the weight vector will be in
  $\bZ^2$ and not $\bZ^4$.
  We will begin with the initial dual graph that is the edge
  $\{v_i, v_j\}$ and we will give the vertices the initial marks 0.

  Now consider a sequence of blowups $Y\to S$ over $P$. Each blowup
  introduces a $(-1)$-curve $E$. On the next surface let us blow up
  one of the two points of intersection of $E$ with the neighboring
  visible curves, either the one on the left or the one on the
  right. The old $(-1)$-curve becomes a $(-2)$-curve and there is a
  new $(-1)$-curve. Then we repeat. Thus, the entire procedure is
  encoded in a binary sequence, such as LRRRLR. Let
  $(w_{i}, w_{j})\in\bZ^2$ be the weight vector of the $(-1)$-curve
  $E$ after the final blowup: $w_{i}$ is the coefficient of $E$ in
  $f^*L_i$ and $w_{j}$ is the coefficient of $E$ in $f^*L_j$.
  Let $\Gamma$ be the final graph; it is a chain.
\end{construction}

\begin{theorem}\label{thm:farey-fractions}
  In Construction~\ref{con:single-curve}, let $\Gamma_i$ be the chain
  on the left of the $(-1)$-curve $E$ in the final graph $\Gamma$, the
  one containing $v_i$. Let $\Gamma_j$ be the chain to the right of $E$,
  the one containing $v_j$. Then
  \begin{enumerate}
  \item $\det\Gamma_i = w_{i}$ and $\det (\Gamma_i - v_i) = w_{j}$.
  \item $\det\Gamma_j = w_{j}$ and $\det (\Gamma_j - v_j) = w_{i}$.
  \end{enumerate}
  In other words, $\Gamma_i$ corresponds to the HJ fraction
  $\frac{w_{j}}{w_{i}}$ and $\Gamma_j$ to $\frac{w_{i}}{w_{j}}$.
  In this way, all the possible visible curves on the blowups over the
  point $P$ are in a bijection with the pairs of coprime positive
  integers $(w_{i}, w_{j})$, or equivalently with 
  the positive
  rational numbers $\frac{w_{j}}{w_{i}}$ written in the simplest form.
\end{theorem}

\begin{examples}\label{ex:weight-vectors}
  (1) The weight vector $(w_{i},w_{j})=(1,1)$ corresponds to a
  single blowup at $P$ and the empty sequence of Ls and Rs. The final
  graph is $\Gamma = [1, \bm{1}, 1]$ with the bold $\bm{1}$
  corresponding to the final $(-1)$-curve $E$.

  \smallskip (2) The weight vector $(w_{i},w_{j})=(n,1)$
  corresponding to the HJ fraction $\frac1{n}$ gives
  $\Gamma = [n, \bm{1}, (n-1)\times 2, 1]$.  The sequence is L\dots{}L
  repeated $(n-1)$ times.

  \smallskip (3)
  The weight vector $(w_{i},w_{j})=(1,n)$
  corresponding to the HJ fraction $\frac{n}{1}$
  gives $\Gamma = [1,(n-1)\times 2, \bm{1}, n]$.
  The sequence is R\dots{}R repeated $(n-1)$ times.

  \smallskip (4) The sequence LRRRLR gives
  $\Gamma = [2,2,2,3,2, \bm{1}, 3,5,1]$, and the weight vector is
  $(w_{i},w_{j})=(14,11)$. The HJ fraction is $\frac{11}{14}$.
\end{examples}

\begin{proof}[Proof of Theorem~\ref{thm:farey-fractions}]
  Blowing up a point corresponds to inserting a vertex on an edge
  between two vertices $(u,v)$. If the weight vectors of $u$ and $v$
  are $(a,b)$ and $(c,d)$ then the new weight vector is $(a+c,
  b+d)$. We see that the sequence of the blowups is the same as the
  well known procedure for the Farey fractions, encoding every
  positive rational number $\frac{w_{j}}{w_{i}}$ in a binary
  sequence of Ls and Rs. 

  By Lemma~\ref{lem:det-expansion}, the determinants of the chains
  transform exactly the right way: for the new graph
  $(a+c, b+d) = (\det \Gamma_i, \det \Gamma_j) = (\det \Gamma_j - v_j,
  \det \Gamma_i - v_i)$.  In particular, $w_{i}$ and $w_{j}$ remain
  coprime. We are done by induction.
\end{proof}

\begin{construction}\label{con:many-curves}
  We now consider a more general case. Let $Y\to (S,P)$ be some sequence
  of blowups at the points of intersection of the strict preimages of
  $L_i,L_j$ and the 
  exceptional divisors. Let $\Gamma$ be the dual graph of the inverse
  image of $L_i+L_j$.
  We will fix several visible curves $E^k$ on $Y$, $k=1,\dots, N$. By
  Theorem~\ref{thm:farey-fractions} each of them corresponds to a pair
  $(w^k_{i},w^k_{j})$.

  Next, we will consider a birational morphism $g\colon Y\to X$ which
  contracts all the curves except for the $E^k$s. Thus, we may assume
  that in there are no $(-1)$-curves between the $E^k$s. If two of the
  curves, say $E^k$ and $E^{k+1}$ have no curves between them then
  at most one of them could be a $(-1)$-curve. This happens iff in the
  Farey procedure one of the fractions follows another, i.e.\ when
  $w^k_{i}w^{k+1}_{j} - w^k_{j}w^{k+1}_{i} = \pm 1$. We allow
  this.

  For each pair $E^k$, $E^l$ of these curves the beginning of the
  binary sequence of Ls and Rs is the same and then they diverge. In
  the chain $\Gamma$ the vertices of the curves $E^k$ come in the
  increasing order of the fractions:
  \begin{math}
    \frac{w^1_{j}}{w^1_{i}} < \frac{w^2_{j}}{w^2_{i}}
    < \dotsb < \frac{w^N_{j}}{w^N_{i}}.
  \end{math}
\end{construction}

\begin{example}
  For two curves with the weight vectors $(4,3)$ and $(14,11)$ the
  final chain is $\Gamma = [2,2,3, \bm{1}, 4,2, \bm{1}, 3,5,1]$. The
  sequences are LRRL and LRRRLR, they diverge after LRR. Since
  $\frac34 < \frac{11}{14}$, the first curve is on the left.
\end{example}

Now let $f\colon Y\to \bP^2$ denote a surface obtained by blowing up
the four-line configuration in $\bP^2$ and $g\colon Y\to X$ a
contraction as in Construction~\ref{con:main}.
The following description is now obvious from the above.
The four initial marks are $-1 = -L_i^2$ for the lines
in the plane.  

\begin{theorem}\label{thm:dual-graph}
  Assume that $g\colon Y\to X$ does not contract any
  $(-1)$-curves, as can always be arranged.
  The dual graph of the visible curves on $Y$ is obtained by starting
  with a complete graph on four vertices with marks $(-1)$ and then
  for every edge $e_{ij}=\{v_i, v_j\}$ for which the point
  $P_{i} = L_i\cap L_j$ is blown up by $f$ grafting on top of
  the vertices $v_i$ (resp. $v_j$) the chains for the HJ
  fractions $\frac{w^k_{j}}{w^k_{i}}$ (resp. $\frac{w^k_{i}}{w^k_{j}}$)
  corresponding to the weight vectors
  $(w^k_{i}, w^k_{j}, 0, 0)$ for the curves $E^k$ over~$P_{ij}$.
  
\end{theorem}

\begin{theorem}\label{thm:singularities}
  Let $\bm{\Gamma}$ be the graph with the vertices $v_i$
  ($i=1\dots 4$) which are not survivors.  The dual graph
  $\Gamma\dsing$ of the singularities of $X$ has two parts
  $\Gamma\dsing = \wh\Gamma \sqcup \Gamma^{\rm edge}$:
  \begin{enumerate}
  \item $\wh\Gamma$ is the hairy graph with the core graph
    $\bm{\Gamma}$ and the marks
    $\bm{n_i} = -1 + \sum_j \frac{w_{i}}{w_{j}}$, $1\le i\le4$, for
    the survivor (if it exists) \emph{closest to}~$v_i$ along the
    edge $\{v_i,v_j\}$.

  \item $\Gamma^{\rm edge}$ consists of the chains that lie entirely
    inside edges, when there are several survivors on the same edge.
\end{enumerate}
\end{theorem}

The determinant of the graph $\wh\Gamma$ can be easily computed by
Theorem~\ref{thm:hairy-graph}. For the singularities in $\Gamma^{\rm
  edge}$, one has the following:

\begin{lemma}\label{lem:edge-det}
  Let $\Gamma$ be a chain between two visible curves on the same edge,
  with the weight vectors $(w_{i}, w_{j})$, $(w'_{i},w'_{j})$ such
  that $\frac{w_{j}}{w_{i}} < \frac{w'_{j}}{w'_{i}}$.  Then 
    \begin{displaymath}
      \det\Gamma =
      \begin{vmatrix}
        w_{i} & w_{j} \\
        w'_{i} & w'_{j} 
      \end{vmatrix}
      = 
      \left(\frac{w_{i}}{w_{j}} - \frac{w'_{i}}{w'_{j}} \right) 
      w_{j}w'_{j}
    \end{displaymath}
\end{lemma}
\begin{proof}
  This can be formally considered to be a special case of
  Theorem~\ref{thm:hairy-graph}, for a hairy graph with a single vertex
  $v_i$ marked 0 and two hairs for the HJ fractions
  $\frac{w_{j}}{w_{i}}$ and $-\frac{w'_{j}}{w'_{i}}$ attached to it.
\end{proof}

\section{A simple formula for $(K_X+B)^2$}
\label{sec:formulas}

We continue working with a pair $(X,\sum b_iB_i)$ as in
Construction~\ref{con:main}. Let $W$, $\oW$, $\hW$ be the weight
matrices defined in \eqref{def:weight-matrix},
\eqref{def:extended-weight-matrix}. In this Section we always assume
that $\rho(X)=1$. Thus, $W$ is $4\times 4$, $\oW$ is $4\times 3$, and
$\hW$ is $5\times 5$.

\begin{definition}
  We will denote by $W_i$ the $3\times 3$ submatrix of $\oW$ obtained
  by removing the $i$-th row. Note that $\det W_i$ is the same as the
  minor of the extended matrix $\hW$ one gets by removing the $i$-th
  row and the $5$th column.
\end{definition}

Let $\{F_s\}$ be the visible curves which are contracted by
$g\colon Y\to X$, and let $\Gamma\dsing$ be the dual graph of this
collection. The negative of the intersection form $(-F_i.F_j)$ is
positive definite. Let us denote by $\Delta$ its determinant. By
Theorem~\ref{thm:singularities} one has $\Delta = \det\Gamma\dsing =
\det\wh\Gamma \cdot \det\Gamma\dedge$, and these are computed by
\eqref{thm:hairy-graph}, \eqref{lem:edge-det}.

\smallskip

Since the lattice $\Pic Y$ is unimodular, the sublattice
$\la F_s\ra^\perp = \bZ H_Y$ is one dimensional, generated by a
primitive integral vector with $H_Y^2 = \Delta$. 
Let
$p\colon \Pic Y \to \bQ H_Y$ be the orthogonal projection. Its image
is $\bZ h_Y$, where $h_Y = H_Y/\Delta$ and $h_Y^2 = \frac1{\Delta}$.
Let $h:= g_* h_Y$. One has $h^2 = \frac1{\Delta}$. By changing the
sign if necessary, we can assume that $h$ is the ample generator of
$(\Pic X)\otimes\bQ$.

\begin{theorem}\label{thm:main-formulas}
  One has the following:
  \begin{enumerate}
  \item The numbers $\det W$ and $(-1)^i \det W_i$ for $1\le i\le 4$ 
    are all nonzero and have the same sign. 
  \item Permuting the rows of $W$ if necessary, we can assume that
    $\det W > 0$. Then in $(\Pic X)\otimes\bQ$ one has
    \begin{displaymath}
      g_* f^*(L_i) = \det W \cdot h
      \text{ and }
      B_i = (-1)^i \det W_i\cdot h \text{ for } 1\le i \le 4.
    \end{displaymath}
  \item Assuming $\det W>0$, one has $K_X + B = \det\hW \cdot h$. In
    particular, $K_X+B$ is ample, numerically zero, or antiample iff
    $\det\hW>0$, $\det\hW=0$, or $\det\hW<0$.
  \item 
    \begin{displaymath}
      (g_* f^*L_i)^2 = \frac{(\det W)^2}{\Delta}, \quad
      B_i^2 = \frac{(\det W_i)^2}{\Delta}, \quad
      (K_X + B)^2 = \frac{(\det\hW)^2}{\Delta}.
    \end{displaymath}
  \end{enumerate}
\end{theorem}
\begin{proof}
  (1) and (2). Let $g_*f^*L_4 = m h$. Then $m$ is the index of the sublattice
  $\la F_s\ra + f^*L_4$ in $\Pic Y$. This is the same as the index of
  the sublattice
  \begin{displaymath}
    \la F_s\ra + f^*L_4 + \la f^*(L_i-L_4) \ra  = 
    \la F_s\ra + \la f^*(L_i) \ra \subset \Vis,
  \end{displaymath}
  where $\Vis = \oplus\bZ E$ is the free $\bZ$-module generated by the
  visible curves. This, in turn is the same as the index of the
  sublattice
  \begin{displaymath}
    \la g_*f^* L_i, \ 1\le i \le 4\ra \subset \oplus_{i=1}^4 \bZ B_i
    \simeq \bZ^4,
  \end{displaymath}
  which is $|\det W|$ by definition. In particular, $\det W\ne 0$. 

  Similarly, if $B_i = m_i h$ then $m_i$ is the index of the
  sublattice $\la F_s\ra + g^*B_i + \la f^*(L_j-L_4) \ra$
   in
  $\Vis$. Up to a sign, this is the determinant of the matrix obtained
  from $\oW$ by removing the $i$-th row, i.e.\ $|\det W_i|$. The signs
  are easy to figure out: $B_i/B_j$ is the ratio of the corresponding
  cofactors $(-1)^i \det W_i$.

  \smallskip

  (3) We note that $f^*(K_{\bP^2}+ \sum_{i=1}^4 L_i)$ is
  $K_Y + \sum E$, where the sum goes over all the visible curves. Thus
  $K_X + \sum_{i=1}^4 B_i = g_*f^*L = \det W\cdot h$. Here, $\det W$
  is the determinant of the extended matrix $\hW(0)$ with the last
  column entries being the log discrepancies $0$. Since $B_i=m_i h$
  with $m_i$ the cofactors of $\hW$ for the $(i,5)$-entry, for the
  pair with arbitrary log discrepancies $c_i=1-b_i$ we get
  $K_X + B = \widehat{m}h$, where $\widehat{m}$ is the determinant of the
  extended matrix $\hW(c_i)$ with the entries $c_i$ in the last
  column, except of course the $(5,5)$-entry is 1.
  Part (4) follows since $h^2=\frac1{\Delta}$.
\end{proof}

\begin{remark}
  Of course $\det\hW$ can also be computed as the determinant of the
  $4\times 4$ matrix $W - (c_1,c_2,c_3,c_4)^t \cdot (1,1,1,1)$.
\end{remark}

\begin{example}
  Consider the surface $Y$ with the visible curves as in the figure
  below. The extended weight matrix for the divisor $K_X$ is written on
  the right. 

  \begin{minipage}[c]{0.49\textwidth}
    \centering
    \includegraphics{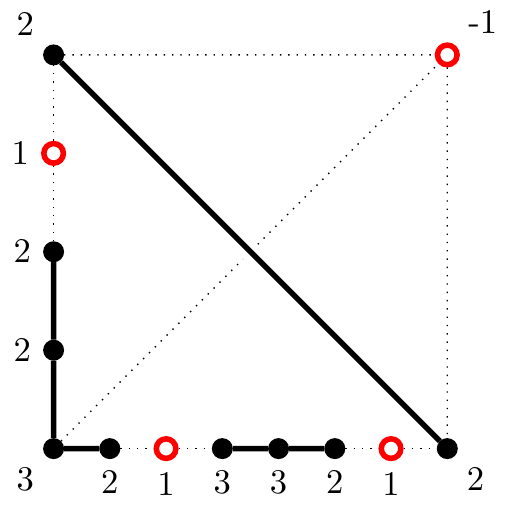}
  \end{minipage}
  \begin{minipage}[c]{0.49\textwidth}
    \centering
    \begin{displaymath}
      \hW =
      \left(
        \begin{array}{cccc|c}
        1 & 3 & 0 & 0 & 1\\
        0 & 0 & 1 & 0 & 1\\
        5 & 0 & 0 & 2 & 1\\
        1 & 0 & 0 & 3 & 1\\
        \hline
        1 & 1 & 1 & 1 & 1
        \end{array}
      \right)
    \end{displaymath}
  \end{minipage}

  \medskip The determinants of the weight matrices are 
  $\det W=39$, $((-1)^i\det W_i) = (13,39,3,11)$,
  $\det\hW = -27$. For the singularities:
  $\Delta_1= |-1+\frac13+\frac52| \cdot 3\cdot 2= 11$,
  $\Delta_2= |-1+\frac31, -1+\frac31| \cdot 1\cdot 1= 3$,
  $\Delta_3 = |\frac52-\frac13| \cdot 2\cdot 3 = 13$. Thus, $-K_X$ is
  ample and $K_X^2 = \frac{(-27)^2}{11\cdot 3\cdot 13} = \frac{243}{143}$.
\end{example}

% \begin{remark}
%   For any fixed configuration of survivors, the second expression in
%   Theorem~\ref{thm:hairy-graph} for $\Delta=\det\Gamma\dsing$ gives a
%   formula for $\Delta$ as the determinant of a $12\times 12$ matrix
%   whose entries are constants or linear homogeneous functions in the
%   weight vectors $\vec w$ of the survivors.  For a configuration with
%   $N$ survivors on the edges, by row and column operations this matrix
%   could be reduced to a $2N \times 2N$-matrix whose entries are linear
%   homogeneous functions of $\vec w$'s.
% \end{remark}

The first formula of Theorem~\ref{thm:hairy-graph} is a very efficient
way to compute $\det \wh\Gamma$ as the determinant of an at most
$4\times 4$ matrix. The formula for $\det \Gamma\dedge$ in
~\eqref{lem:edge-det} is also very simple.
However, the following Lemma is still of an independent interest.

\begin{lemma}\label{lem:det-sing-12}
  Assume that the four survivors are on the edges and correspond to
  the weight vectors $(w^k_{i}, w^k_{j})\in\bZ_i\times\bZ_j$, $k=1,\dotsc,4$,
  and $(i,j)$ depend on $k$. 
  Then
  \begin{displaymath}
    \det \Gamma\dsing =
    \det \bm{\wt\Gamma} \cdot
    \prod_{k=1}^4 w^k_{i}w^k_{j}, 
  \end{displaymath}
  where $\bm{\wt\Gamma}$ is a graph on 12=4+8 vertices, as follows:
  \begin{enumerate}
  \item The first 4 vertices $v_i$ correspond to the lines $L_i$ and have
    marks $-1$. 
The two vertices $v_i,v_j$ are connected by an edge $e_{ij}$
    iff there are no survivors on this edge, i.e. if the point
    $L_i\cap L_j$ is not blown up by $f\colon Y\to\bP^2$. 
  \item Each of the survivors $E^k$ on the edge $e_{ij}$ gives two vertices with the marks
    $-w^k_j/w^k_i$ and $-w^k_i/w^k_j$. The vertex $v_i$ is
    connected to the vertex with the mark $-w^k_j/w^k_i$ for
    the survivor closest to $v_i$. Then the
    vertex with the mark $-w^k_i/w^k_j$ is connected to the next
    survivor on the edge $e_{ij}$ if it exists, or to the vertex $v_j$ if it does not.
  \end{enumerate}
  By clearing the denominators, this gives an expression for
  $\Gamma\dsing$ as the determinant of a $12\times 12$-matrix with 4
  rows with constant entries and 8 rows in which entries are 0 or
  $w^k_i$. By elementary column operations, it can be reduced to a
  determinant of an $8\times 8$ matrix whose entries are linear
  functions of $w^k_i$.
\end{lemma}
\begin{proof}
  This follows immediately from the second formula of
  Theorem~\ref{thm:hairy-graph}.
\end{proof}

\begin{example}
  For the case 3 of Theorem~\ref{thm:6-cases} below,
  $\det\Gamma\dsing$ is the determinant of the following $12\times 12$
  matrix, in which for convenience the call $w^k_i$ by $a_i$, $b_i$, $c_i$,
  $d_i$ and the dots denote zeros. Also, since all the entries in the
  matrix $M$ for the graph $\bm{\wt\Gamma}$ in \eqref{lem:det-sing-12} are
  nonpositive and $\det(-M)=\det(M)$, we use $-M$.

  \medskip
  
  \begin{displaymath}
    \left|
    \begin{array}{cccccccccccc}
      1 & . & 1 & . & 1 & . & . & . & 1 & . & . & . \\
      . & 1 & . & 1 & . & . & . & 1 & . & . & 1 & . \\
      1 & . & 1 & 1 & . & . & . & . & . & . & . & 1 \\
      . & 1 & 1 & 1 & . & . & . & . & . & 1 & . & . \\
      a_1 & . & . & . & a_2 & . & . & . & . & . & . & . \\
      . & . & . & . & . & a_1 & a_2 & . & . & . & . & . \\
      . & . & . & . & . & b_1 & b_2 & . & . & . & . & . \\
      . & b_2 & . & . & . & . & . & b_1 & . & . & . & . \\
      c_1 & . & . & . & . & . & . & . & c_4 & . & . & . \\
      . & . & . & c_4 & . & . & . & . & . & c_1 & . & . \\
      . & d_2 & . & . & . & . & . & . & . & . & d_3 & . \\
      . & . & d_3 & . & . & . & . & . & . & . & . & d_2
    \end{array}
    \right|
  \end{displaymath}
\end{example}

% \bigskip
\newpage

\section{Minimal volumes of surfaces with ample $K_X$} 
\label{sec:min-volumes}

We keep the notations of Construction~\ref{con:main}.

\begin{lemma}\label{lem:coef-1}
  Assume that $(X,B)$ is log canonical and that for one of the curves
  $b_i=1$. Then $B_i$ is an image of $g_*\inv L_i$, i.e.\ of one of the
  four original lines in $\bP^2$.
\end{lemma}
\begin{proof}
  Let $B_i^\nu$ be the normalization of $B_i$. By adjunction,
  $(K_X+B)|_{B_i^\nu} = K_{B_i^\nu} + \Diff$, where
  $\Diff = \sum d_k P_k$ is the different. Since $(X,B)$ is log
  canonical, $d_k\le1$. Since $B_i=\bP^1$, $\deg K_{B_i^\nu}=-2$ and
  one must have at least 3 points $P_k$. Thus, $B_i$ comes from a
  corner vertex in the graph, i.e.\ from one of the $L_i$'s.
\end{proof}

\begin{lemma}\label{lem:coef-0}
  Assume that $\rho(X)=1$.  Let $B_i$ be not a corner, i.e.
  $B_i\ne g_*f_*\inv L_i$.  Let $B' := B_i + \sum_{j\ne i} b_jB_j$.
  Then the pair $(X,B')$ is \emph{not} log canonical at at least one
  point of $X$ lying on $B_i$.
\end{lemma}
\begin{proof}
  Since $B'\ge B$ and $\rho(X)=1$, one has $(K_X+B')B_i > 0$.
  Adjunction gives $(K_X+B')_{B_i^\nu} = K_{B_i^\nu} + \sum d_kP_k$
  which has degree $-2+\sum d_k$.  If $B_i$ is not a corner then on the
  normalization $B_i^\nu$ there are at most two points $P_k$. So for one
  of them $d_k>1$ and the pair $(X,B')$ is not log canonical at that
  point.
\end{proof}

\begin{corollary}\label{cor:no-middle-survivor}
  For a log canonical pair $(X,B)$ with $\rho(X)=1$, one can not have
  three survivors on the same edge.
\end{corollary}
\begin{proof}
  For the middle survivor the pair $(X,B')$ of the previous lemma is
  log canonical by the classification we recalled in the introduction,
  since the singularities on both sides correspond to chains.
\end{proof}

\begin{theorem}\label{thm:6-cases}
  For the set $\cS_0=\{0\}$, i.e.\ for the log canonical surfaces $X$
  with ample $K_X$, there are 6 possibilities for the position of the
  survivors in the graph, given in Fig.~\ref{fig:6-cases}. In
  particular, all the survivors are on the edges, and none of them are
  in the corners.
\end{theorem}

\begin{figure}[h!]
  \centering
  \includegraphics[width=\textwidth]{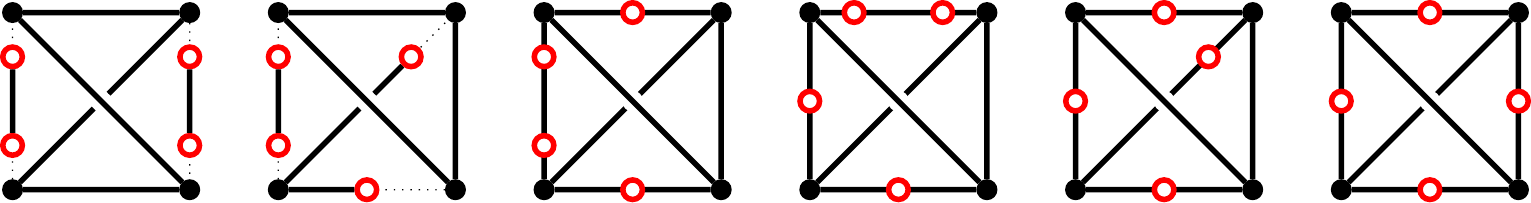}
  \caption{Log canonical surfaces with ample $K_X$}
  \label{fig:6-cases}
\end{figure}

\begin{proof}
  This is a straightforward enumeration of the cases. There are only 8
  cases satisfying Corollary~\ref{cor:no-middle-survivor}. One of them
  is $\bP^2$ with 4 lines, so that $-K_X$ is ample, and another one has
  $K_X=0$. The 6 listed cases are all legal and do appear. 
\end{proof}

\begin{lemma}\label{lem:survivor}
Let $f\colon Y\rightarrow \PP^2$ be a sequence of blow-ups over the nodes of four lines $L=\sum L_i$, and $g\colon Y\rightarrow X$ the contraction of some visible curves, including $f^{-1}_*L$ but not any of the $(-1)$-curves, such that $K_X$ is log canonical and ample. Then each survivor is the image of a $(-1)$-curve on $Y$.
\end{lemma}
\begin{proof}
Let $B_0\subset X$ be a survivor and $E_0$ its strict transform of $B_0$ on $Y$. Since $g$ contracts $f_*^{-1}L$, the curve $E_0$ is $f$-exceptional. To see that $E_0$ is a $(-1)$-curve, it suffices to show that there is no other $f$-exceptional curve over $E_0$: otherwise let $\pi\colon Y\rightarrow Y'$ be the contraction of all the $f$-exceptional curves over $E_0$. Then $Y'$ is smooth and $f\colon Y\rightarrow \PP^2$ factors through some morphism $f'\colon Y'\rightarrow \PP^2$. The log canonical divisor $\pi_*(g^*K_X)$ is canonical along the divisor $E'_0 := \pi_* E_0$. It follows that $\pi^*\pi_* (g^*K_X)=g^* K_X$, and hence $g\colon Y\rightarrow X$ factors through $\pi\colon Y\rightarrow Y'$, contradicting the assumption that $g$ is the minimal resolution.
\end{proof}

\begin{lemma} \label{lem:K2-increases} 
  Let $X_1,X_2$ be two log canonical surfaces with ample canonical
  class, and let $g_n\colon Y_n\to X_n$ be their minimal resolutions
  ($n=1,2$). Assume that there exists a (non identity) morphism $\pi:Y_2\to Y_1$
  mapping the four survivors $E_2^i$ to the four survivors $E_1^i$
  ($1\leq i\leq 4$) in such a way that the LR sequence for each $E_2^i$ prolongs that
  of $E_1^i$. Then one has $K_{X_2}^2 > K_{X_1}^2$. 
\end{lemma}

\begin{proof}
  Let $C_n$ be the union of all the curves contracted by $g_n$
  ($n=1,2$). Then $\vol(K_{X_n}) = \vol(K_{Y_{n }}+ C_n)$ and
  $C_1< \pi_* C_2$.
  Since $(Y_1 , C_1)$ is canonical at the points blown up by $\pi$ and
  $ C_1<\pi_* C_2 $, one has
  \begin{math}
    \pi^*(K_{Y_1}+C_1) \le K_{Y_2} + \pi_*\inv C_1 <
    K_{Y_2} + C_2.
  \end{math}
  It follows that
  \[\vol(K_{X_1}) = \vol(K_{Y_1} +C_1) <
    \vol(K_{Y_2}+C_2) =\vol(K_{X_2}).\]
\end{proof}

\begin{theorem}\label{thm:min-k2}
  In the 6 cases of Fig.~\ref{fig:6-cases}, the minimal $K_X^2$ are as
  in Table~\ref{tab:min}, achieved for the listed weight matrices. In
  particular, the absolute minimum is $1/6351$.
\end{theorem}
\begin{table}[h!]
  \centering
  \begin{tabular}{ccl}
    Case&Min $K^2$&Achieved at the weight matrix $W$\\
    1 & 1/143  &  [[2, 1, 0, 0], [1, 7, 0, 0], [0, 0, 3, 1], [0, 0, 1, 4]] \\
    2 & 1/143  &  [[2, 1, 0, 0], [1, 7, 0, 0], [1, 0, 2, 0], [1, 0, 0, 3]] \\
    3 & 1/5537  &  [[5, 1, 0, 0], [1, 10, 0, 0], [1, 0, 0, 3], [0, 1, 2, 0]] \\
    4 & 1/5537  &  [[2, 1, 0, 0], [1, 0, 0, 2], [0, 10, 1, 0], [0, 1, 5, 0]] \\
    5 & 1/6351  &  [[1, 2, 0, 0], [9, 0, 1, 0], [1, 0, 0, 5], [0, 1, 2, 0]] \\
    6 & 1/6351  &  [[1, 2, 0, 0], [0, 1, 2, 0], [0, 0, 1, 4], [10, 0, 0, 1]] 
  \end{tabular}
  \caption{Minimum $K_X^2$ in the 6 cases}
  \label{tab:min}
\end{table}

\begin{proof}
  At each of the corners in Fig.~\ref{fig:6-cases} one can have a
  singularity with a fork, of type $D$ or $E$. However, by the
  classification recalled in the introduction, the cases for the
  determinants of the chains out of the fork are 1, 2, 3, 4, 5, 6, or
  $n\ge7$, and the possibilities for any $n\ge7$ are the same as for $7$.

  We thus have finitely many possibilities for the weights $1\le
  w_{i}, w_{j} \le 7$ on each edge $\{v_i,v_j\}$.
  For each of them and for each fork, we have a
  condition that the singularity must be log canonical. The formula
  for the log discrepancy at a vertex was given in
  \cite{alexeev1992log-canonical-surface} and is as follows
  (here, $\deg(v)$ is the valency of the vertex $v$): 
  \begin{displaymath}
    c(u) = \frac1{\det\Gamma} \sum_{v \in {\rm Vert}(\Gamma)}
    \big( 2- \deg(v) \big) \cdot 
    \det \big( \Gamma - {\rm path}(u,v) \big)
  \end{displaymath}
  For log canonical, one must have $c(u)\ge0$ for each of the 4
  corners in the graph.  Using this formula, for each of the cases, we
  get finitely many series that depend on $0\le p \le 4$ parameters
  $x_i$. In cases 1 and 2 there is only one series up to symmetry,
  case 3: 2, case 4: 3, case 5: 60, and case 6: 18 series,
  for a total of 85 series.

  Lemma~\ref{lem:K2-increases} allows us to reduce the proof to
  checking finitely many cases. In each series the weight vectors of
  the survivors are either constant or are of the form $(n,k)\in\bZ_i\times\bZ_j$
  where $k\in\{1,\dotsc,6\}$ is fixed and $n\to\infty$, subject to the condition
  that $(n,k)=1$. If $k=1$ then as in
  Example~\ref{ex:weight-vectors}(2), the LR sequence is $L^{n-1}$. So
  once $K_X$ is ample for a certain surface in the series, increasing
  $n$ only makes $K_X^2$ larger.

  If $k=2$ then the LR sequence for the weight $(2n-1,2)$ is
  $L^{n-1}R$. This is preceded by the sequence $L^{n-1}$ for the
  weight $(n,1)$. Once the canonical class for the latter sequence is
  ample, all the other surfaces obtained by increasing $n$ in the weight
  $(2n-1,2)$ will have a larger volume.
  Note also  that if the surface for $(2n-1,2)$ is log canonical then
  so is the surface for $(n,1)$.

  For the weight $(3n-1,3)$ the
  sequence is $L^{n-1}RL$, and for $(3n-2,3)$ it is $L^{n-1}R^2$. Once
  again, these are preceded by a log canonical surface with the weight
  $(n,1)$ and 
  once for large enough $n$ the latter surface has ample $K_X$, the rest of
  the series is redundant.
  The cases $k=4,5,6$ are done entirely similarly.
  We are thus reduced to finitely many cases, which we checked using
  Theorem~\ref{thm:main-formulas} and a sage \cite{sagemath} script. This concludes the proof.
  Even though it is redundant, below is an
  alternative way to reduce to finitely many checks.

  \smallskip

  As above, we get finitely many series of surfaces appearing in cases
  1--6. Let us work with one of them: $\{X(n_1,\dotsc, n_p)\}$,
  depending on $p\le 4$ parameters. There are only finitely many
  minimal, in the lexicographic order, sequences $(n_1,\dotsc, n_p)$
  for which $K_X$ is ample, i.e. $\det\hW(n_1,\dotsc,n_p)>0$ in
  Theorem~\ref{thm:main-formulas}.  We claim that it is sufficient to
  seek the minimum $K_X^2$ among these minimal sequences plus a few
  more.  By Lemma~\ref{lem:K2-increases}, for each survivor of the
  form $(n,1)$, increasing $n$ makes $K^2$ larger.
  By looking at the 85 series, we
  observe that at most one of the weight vectors $(n_s, k_s)$
  has $k_s\ge 2$, say $k_1\ge2$. We deal with this vector differently.

  Let us denote $x=n_1$. By \eqref{thm:main-formulas} the function
  $f(x)= K^2(x)$ up to a constant has the form
  $\frac{(x-a)^2}{(x-b)^2 + c}$. From the fact that in
  Theorem~\ref{thm:6-cases} no surface with ample $K_X$ has survivors in the corners,
  by row expansion of $\det\hW$ it follows that $a\ge0$. By the
  general theory of \cite{alexeev1994boundedness-and-ksp-2},
  the function $K^2(x)$ is increasing for $x\gg0$. This gives $a\ge
  b$. By computing the derivative $f'(x)$ one easily sees that if
  $f(x+1)\ge f(x)$ then $f(y+1) \ge f(y)$ for any $y\ge x$. Thus, 
  for each of the minimal sequences $(n_1,n_2,\dotsc,n_p)$ it
  suffices to check that $K^2(n_1+1,n_2,\dotsc,n_p) \ge
  K^2(n_1,n_2,\dotsc,n_p)$. We performed this check as well.
\end{proof}

\begin{remark}
  For the best surface in case 2, the surface $Y$ as in \eqref{con:main}
  is obtained from the one in case 1 by
  contracting a $(-1)$-curve. Thus, in fact the surfaces $X$ with ample
  $K_X$ are the
  same. We showed in \cite{alexeev2016open-surfaces} that the surfaces
  in cases 5 and 6 are isomorphic, only the presentations with the
  visible curves are different. Similarly, one can show that the best
  surfaces in cases 3 and 4 are isomorphic.
\end{remark}

\begin{remark}
  In each of the subcases of the main six cases, the series depend on
  $0\le p\le 4$ parameters. The series with the maximal number of 4 parameters are given
  in Table~\ref{tab:4-param-series} and depicted in
  Fig.~\ref{fig:4-param-series}.

\begin{figure}[h!]
  \centering
  \includegraphics[width=\textwidth]{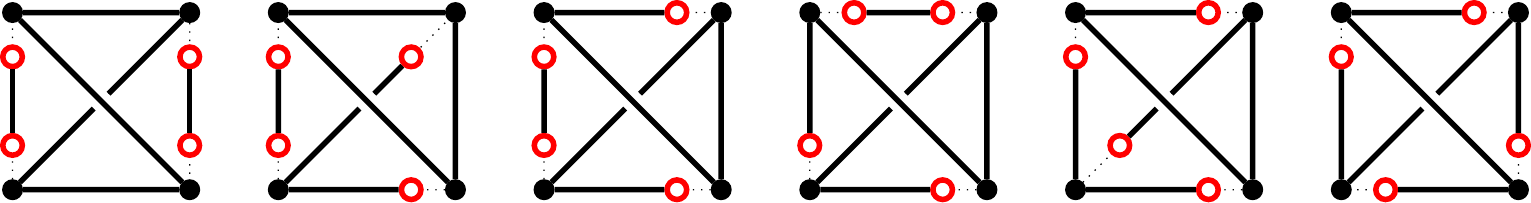}
  \caption{The 4-parameter series}
  \label{fig:4-param-series}
\end{figure}

\begin{table}[h!]
  \centering
  \begin{tabular}{ccl}
    Case & Weight matrix $W$\\
    1 & $[[x_1, 1, 0, 0],\ [1, x_2, 0, 0],\ [0, 0, x_3, 1],\ [0, 0, 1, x_4]]$ \\
    2 & $[[x_1, 1, 0, 0],\ [1, x_2, 0, 0],\ [1, 0, x_3, 0],\ [1, 0, 0, x_4]]$ \\
    3 & $[[x_1, 1, 0, 0],\ [1, x_2, 0, 0],\ [1, 0, 0, x_3],\ [0, 1, x_4, 0]]$ \\
    4 & $[[x_1, 1, 0, 0],\ [1, 0, 0, x_2],\ [0, x_3, 1, 0],\ [0, 1, x_4, 0]]$ \\
    5 & $[[1, x_1, 0, 0],\ [x_2, 0, 1, 0],\ [1, 0, 0, x_3],\ [0, 1, x_4, 0]]$ \\
    6 & $[[1, x_1, 0, 0],\ [0, 1, x_2, 0],\ [0, 0, 1, x_3],\ [x_4, 0, 0, 1]]$
  \end{tabular}
  \caption{Weight matrices in the 4-parameter series}
  \label{tab:4-param-series}
\end{table}
  
In these 4-parameter series, all singularities are of the $A$ type,
i.e.\ correspond to chains only. (In the series with fewer parameters,
\emph{forks do appear.})
Using Corollary~\ref{cor:attaching-2s}, here are the explicit formulas
for the determinants of the singularities:
\begin{enumerate}
\item $\Delta= |x_1,x_2| \cdot |x_3, x_4| \cdot
  |x_1-1, x_3-1, x_2-1, x_4-1, \circlearrowleft|$.
\item $\Delta= |x_1, x_2| \cdot |-x_3, x_1-1, -x_4| \cdot
  |x_2-1, x_3-1, x_4-1, \circlearrowleft|$.
\item $\Delta= |x_1, x_2| \cdot |-x_4, x_1-1, x_3-1, x_4-1, x_2-1, -x_3|$.
\item $\Delta= |x_2, x_3| \cdot |-x_4, x_1-1, x_3-1, x_4-1, x_2-1, -x_1|$.
\item $\Delta= |-x_1, x_3-1, -x_4| \cdot |-x_2, x_1-1, x_4-1, x_2-1,
  -x_3|$.
\item $\Delta= |-x_1, x_4-1, x_2-1, -x_3| \cdot |-x_2, x_1-1, x_3-1,
  -x_4|$. 
\end{enumerate}

In every series, both the numerator and denominator in $K_X^2 =
\frac{(\det\hW)^2}{\Delta}$ is a polynomial of multidegree $(2,2,2,2)$ in the
variables $x_1,x_2,x_3, x_4$ with the leading term
$x_1^2x_2^2x_3^2x_4^2$, and the limit of $K_X^2$ as all $x_i\to\infty$
is 1.
\end{remark}

\section{Pairs $(X,B)$ with reduced $B$, and limit points of volumes}
\label{sec:reduced-B}

\begin{theorem}\label{thm:min-k2b}
  For the the log canonical pairs $(X,B)$ with reduced nonempty
  divisor~$B$
  there are the 12 cases of Fig.~\ref{fig:12-cases}, plus
  $(\bP^2,\sum_{k=1}^4L_k)$. The minimal $(K_X+B)^2$ for these
  cases are as in Table~\ref{tab:min-S1}, achieved for the listed
  weight matrices.  In particular, the absolute minimum for
  the volume in these settings is $1/78$.
\end{theorem}

The proof is the same as for Theorem~\ref{thm:min-k2}.
 
\begin{figure}[h!]
  \centering
  \includegraphics[width=\textwidth]{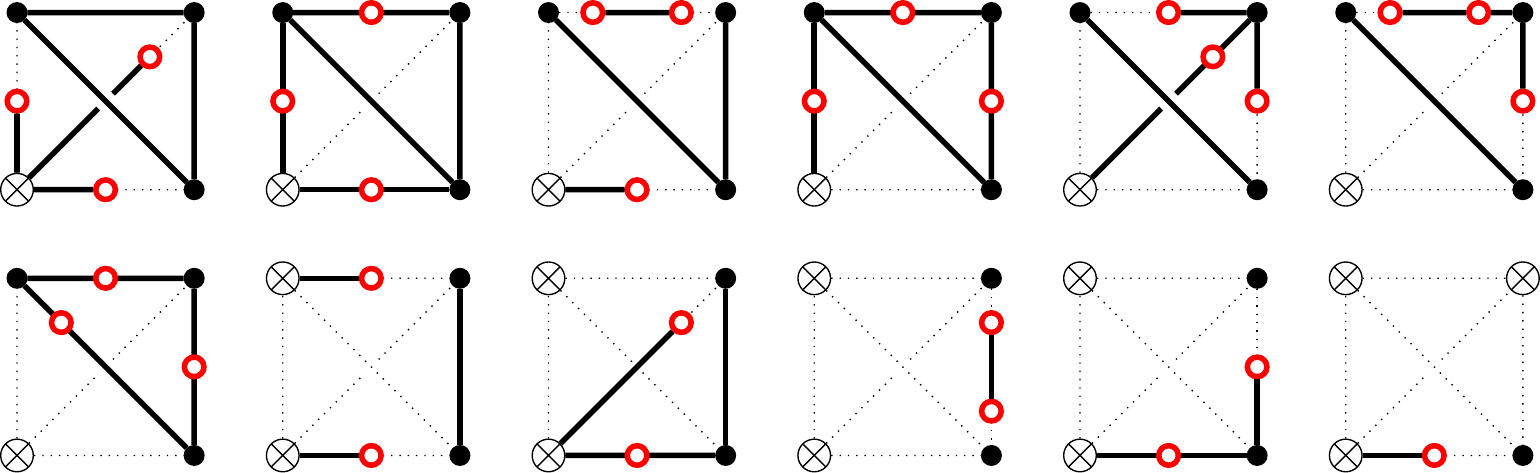}
  \caption{Log canonical pairs with ample $K_X+B$, reduced $B\ne0$}
  \label{fig:12-cases}
\end{figure}

\begin{table}[h!]
  \centering
  \begin{tabular}{ccl}
    Case&Min $(K+B)^2$&Achieved at the weight matrix $W$\\
    1&1/42  &  [[1, 0, 0, 0], [1, 2, 0, 0], [1, 0, 3, 0], [1, 0, 0, 7]] \\
    2& 1/78  &  [[1, 0, 0, 0], [1, 2, 0, 0], [1, 0, 0, 3], [0, 1, 4, 0]] \\
    3& 1/22  &  [[1, 0, 0, 0], [1, 0, 0, 2], [0, 3, 1, 0], [0, 1, 4, 0]]\\
    4&1/70  &  [[1, 0, 0, 0], [1, 2, 0, 0], [0, 1, 2, 0], [0, 0, 1, 4]] \\
    5& 1/22  &  [[1, 0, 0, 0], [1, 0, 2, 0], [0, 2, 1, 0], [0, 0, 1, 3]] \\
    6& 1/15  &  [[1, 0, 0, 0], [0, 3, 1, 0], [0, 1, 2, 0], [0, 0, 1, 2]] \\
    7& 1/60  &  [[1, 0, 0, 0], [0, 1, 2, 0], [0, 2, 0, 1], [0, 0, 1, 3]] \\
    8& 1/6  &  [[1, 0, 0, 0], [0, 1, 0, 0], [1, 0, 0, 2], [0, 1, 3, 0]] \\
    9& 1/6  &  [[1, 0, 0, 0], [0, 1, 0, 0], [1, 0, 2, 0], [1, 0, 0, 3]] \\
    10& 1/3  &  [[1, 0, 0, 0], [0, 1, 0, 0], [0, 0, 2, 1], [0, 0, 1, 2]] \\
    11& 1/6  &  [[1, 0, 0, 0], [0, 1, 0, 0], [1, 0, 0, 2], [0, 0, 2, 1]] \\
    12& 1/2  &  [[1, 0, 0, 0], [0, 1, 0, 0], [0, 0, 1, 0], [1, 0, 0, 2]] \\
    13 & 1 &      [[1, 0, 0, 0], [0, 1, 0, 0], [0, 0, 1, 0], [0, 0, 0, 1]]
  \end{tabular}
  \caption{Minimum $(K_X+B)^2$ in the 13 cases}
  \label{tab:min-S1}
\end{table}

\begin{theorem}
  Let $(X_n,B_n)$, $n\ge n_0$ be a series of log canonical pairs with
  ample $K_{X_n}+B_n$ in which one of the survivors has the weight
  vector $\vec w(n)=(n,k,0,0)$ with $n\to\infty$, and the other weights and
  log discrepancies $c_i$ of the survivors are fixed.  Then the limit of the volumes
  $(K_{X_n}+B_n)^2$ is $(K_{\oX}+\oB)^2$ where the pair $(\oX,\oB)$ is
  obtained by replacing $\vec w(n)$ by $(1,0,0,0)$ and setting the
  log discrepancy $\bar c_1=0$. In other words, $L_1$ is a survivor for
  $\oX$ and it appears in $\oB$ with coefficient $b_1=1$.
\end{theorem}

\begin{proof}
  By Theorem~\ref{thm:main-formulas}, we have
  $(K_X+B)^2= \frac{(\det\hW)^2}{\Delta}$.  The function
  $\det\hW(X_n)$ is linear in $n$, with the leading coefficient equal to the
  determinant of the matrix obtained by replacing the row
  $(n,k,0,0; c)$ by $(1,0,0,0; 0)$.  The determinant $\Delta(X_n)$ for
  the singularities is a quadratic function of $n$ and it easily
  follows from the formulas in~\eqref{thm:hairy-graph},
  \eqref{lem:edge-det} that the coefficient of $n^2$ is
  $\Delta(\oX)$. Thus,
  \begin{displaymath}
    \lim_{n\to\infty}
    \frac{\big(\det\hW(X_n)\big)^2}{\Delta(X_n)} =
    \frac{\big(\det\hW(\oX)\big)^2}{\Delta(\oX)}.
  \end{displaymath}

\end{proof}

\begin{corollary}
  The smallest limit point for the log canonical pairs $(X,B)$ with
  coefficients in $\{0,1\}$ with $\rho(X)=1$ obtained from the four-line
  configuration is~$1/78$.
\end{corollary}
\begin{proof}
  Indeed, the minimal volume $\frac1{78}$ in case 2 of
  Table~\ref{tab:min-S1} appears as the limit of the volumes in case
  5 of Table~\ref{tab:4-param-series} for $x_1=2$, $x_3=3$, $x_4=4$
  and $x_2\to\infty$.
\end{proof}

We conclude with the following:

\begin{lemma}
  The set $\bK^2 = \{K_X^2\}$ of volumes of log canonical surfaces
  obtained via Construction~\ref{con:main} has accumulation complexity
  4, i.e. $\Acc^4(\bK^2) \ne\emptyset$,
  $\Acc^5(\bK^2)=~\emptyset$, where $\Acc^0(\bK^2) = \bK^2$ and
  $\Acc^{n+1}(\bK^2)$ is the set of accumulation points of
  $\Acc^n(\bK^2)$.
\end{lemma}
\begin{proof}
  Indeed, in the proof of Theorem~\ref{thm:min-k2} we produced
  finitely many (85 to be exact) series of surfaces $X(n_1,\dotsc,n_p)$ that depend on
  $p\le 4$ integer parameters. Sending any of $n_i\to\infty$ gives an
  accumulation point, sending $n_j\to\infty$ for $n_j\ne n_i$ gives a
  point in $\Acc^2(\bK^2)$, etc.
\end{proof}

\bibliographystyle{amsalpha}
% \bibliography{va}

\def\cprime{$'$}
\providecommand{\bysame}{\leavevmode\hbox to3em{\hrulefill}\thinspace}
\providecommand{\MR}{\relax\ifhmode\unskip\space\fi MR }
% \MRhref is called by the amsart/book/proc definition of \MR.
\providecommand{\MRhref}[2]{%
  \href{http://www.ams.org/mathscinet-getitem?mr=#1}{#2}
}
\providecommand{\href}[2]{#2}

\end{document}